\documentclass{amsart}
\usepackage{amssymb}

\newtheorem{thm}{Theorem}[section]
\newtheorem{lem}[thm]{Lemma}
\newtheorem{cor}[thm]{Corollary}
\newtheorem{prop}[thm]{Proposition}
\newtheorem{prob}[thm] {Problem}

\theoremstyle{definition}
\newtheorem{defn}[thm]{Definition}

\theoremstyle{remark}

\newcommand{\sP}{\mathcal{P}}

\newcommand{\R}{\mathbb{R}}

\newcommand{\zf}{\text{ZF}}
\newcommand{\ch}{\text{CH}}

\newcommand{\sbset}{\subseteq}
\newcommand{\card}{\mathit{card}}

\def\({\left(}
\def\){\right)}

\begin{document}

\title{CH, $\mathbf{V}=\mathbf{L}$, Disintegration of measures, and $\mathbf{\Pi}^1_1$ sets}

\author{Karl Backs}
\address{Department of Mathematics\\University of North Texas\\
Denton, TX 76203} 
\email{karl.backs@unt.edu}

\author{Steve Jackson*}
\address{Department of Mathematics\\University of North Texas\\
Denton, TX 76203}
\email{jackson@unt.edu}
\thanks{*Research supported by NSF Grant DMS-1201290}

\author{R. Daniel Mauldin**}
\address{Department of Mathematics\\University of North Texas\\
Denton, TX 76203}
\email{mauldin@unt.edu}
\thanks{**Research supported by NSF Grants DMS-0700831
and DMS-0652450}

\date{\today}



\begin{abstract}

In 1950 Maharam asked whether every disintegration of a
$\sigma$-finite measure into $\sigma$-finite measures is necessarily
uniformly $\sigma$-finite. Over the years under special conditions on
the disintegration, the answer was shown to be yes. However, we show
here that the answer may depend on
the axioms of set theory in the following sense. If CH, the continuum
hypothesis holds, then the answer is no. One proof of this leads to some
interesting problems in infinitary combinatorics. If G\"odel's axiom of
constructibility $\mathbf{V}=\mathbf{L}$ holds, then not only is the
answer no, but, of equal interest is the construction of $\mathbf{\Pi}^1_1$ sets with
very special properties.

\end{abstract}

\maketitle

\section{Introduction and Background} \label{sec:one}

Disintegration of a measure has long been a very useful tool in ergodic theory
(see, for examples, \cite{JvN} and \cite{AA}) and
in the theory of conditional probabilities \cite {Pa}.
The origins of disintegration are hazy but the first rigorous
definitions and results seem to be due to von Neumann \cite {JvN}. We
recall the formal definiton of a disintegration considered in this paper.

Let $(X,\mathcal{B}(X))$ and $(Y, \mathcal{B}(Y))$ be uncountable
Polish spaces each equipped with the $\sigma$-algebra of Borel sets, let
$\phi:X \rightarrow Y$ be measurable, and let $\mu$ and $\nu$ be
measures on $\mathcal{B}(X)$ and $\mathcal{B}(Y)$ respectively.

\begin{defn}
A \textbf{disintegration} of $\mu$ with respect to $(\nu,\phi)$ is a family, $\{\mu_y:y \in Y\}$, of measures on $(X,\mathcal{B}(X))$ satisfying:
\begin{enumerate}
    \item
        $\forall B \in \mathcal{B}(X),\,\, y \mapsto \mu_y(B) \text{ is } \mathcal{B}(Y) \text{-measurable}$
    \item
        $\forall y \in Y,\,\, \mu_y(X \setminus \phi^{-1}(y)) = 0$ and
    \item
        $\forall B \in \mathcal{B}(X),\,\, \mu(B) = \int \mu_y (B) d\nu(y)$.
\end{enumerate}
\end{defn}

One could consider disintegrations in more general settings but we will
consider only this setting or the setting where
$X$ and $Y$ are standard Borel spaces, \emph{i.e.}, measure spaces isomorphic to uncountable Polish spaces equipped with the $\sigma$-algebra of Borel sets.

Let us recall that if $\{\mu_y : y \in Y\}$ is a disintegration of
$\mu$ with respect to $(\nu,\phi)$, then the image measure, $\mu \circ
\phi^{-1}$, is absolutely continuous with respect to $\nu$ in the
following sense. If $N \in \mathcal{B}(Y)$ with $\nu(N)=0$, then combining properties (2) and (3) we have
\begin{align*}
\mu \circ \phi^{-1}(N) &= \int \mu_y(\phi^{-1} N) d\nu(y) \\
    &= \int_N \mu_y(\phi^{-1} N) d\nu(y) \\
    &= 0.
\end{align*}
The converse also holds in our setting, (see, for example, \cite{fabec}).

\begin{thm}
Suppose $(X,\mathcal{B}(X))$ and $(Y,\mathcal{B}(Y))$ are standard
Borel spaces, $\mu$ is a $\sigma$-finite measure on $\mathcal{B}(X)$,
$\nu$ is a $\sigma$-finite measure on $\mathcal{B}(Y)$, and $\phi:X
\rightarrow Y$ is a Borel measurable function. If $\mu \circ \phi^{-1}
<< \nu$ then there exists a $\sigma$-finite disintegration $\{\mu_y :
y \in Y\}$ of $\mu$ with respect to $(\nu,\phi)$. Moreover this
disintegration is unique in the sense that if $\{\hat{\mu}_y : y \in
Y\}$ is any $\sigma$-finite disintegration of $\mu$ with respect to
$(\nu,\phi)$, then there exists $N \sbset Y$ such that $\nu(N) = 0$
and $\forall y \not \in N \,\, \mu_y = \hat{\mu}_y$.
\end{thm}

In the late 1940's Rokhlin \cite{Ro} and independently, Maharam \cite{Mah1}
introduced canonical representations of disintegrations of a finite
measure into finite measures. This situation naturally arises when one
is considering a dynamical system with an invariant finite
measure or when one obtains the conditional probability distribution
induced by a given probability measure. Maharam also considered disintegrations
of $\sigma$-finite measures. This situation arises when one
has a dynamical system with a $\sigma$-finite invariant measure, but no
finite invariant measure (see, for example, \cite{EHW}). In her investigation of $\sigma$-finite
disintegrations, Maharam found a basic problem which does not occur in the case of disintegrations of a finite
measure. To explain this problem we make the following definitions.

\begin{defn}
If $\{\mu_y:y \in Y\}$ is a disintegration of $\mu$ with respect to $(\nu,\phi)$ such that $\forall y \in Y$, $\mu_y$ is $\sigma$-finite, then we say that the disintegration is \textbf{$\sigma$-finite}. If $\{\mu_y:y \in Y\}$ is a $\sigma$-finite disintegration of $\mu$ with respect to $(\nu,\phi)$ we say that the disintegration is \textbf{uniformly $\sigma$-finite} provided there exists a sequence, $(B_n)$, from $\mathcal{B}(X)$ such that
\begin{enumerate}
    \item
    $\forall n \in \mathbb{N} \,\, \forall y \in Y, \,\, \mu_y(B_n) < \infty$ and
    \item
    $\forall y \in Y, \,\, \mu_y(X \setminus \bigcup_n B_n) = 0$.
\end{enumerate}
\end{defn}

\begin{prob}
Maharam \ \cite{Mah1,Mah2}: Let $\{\mu_y:y \in Y\}$ be a $\sigma$-finite disintegration
of $\mu$ with respect to $(\nu,\phi)$. Is this disintegration
uniformly $\sigma$-finite?
\end{prob}

The following theorem demonstrates in what manner a given
disintegration is ``almost'' uniformly $\sigma$-finite.

\begin{thm}\label{almostsf}
Suppose $\{\mu_y:y \in Y\}$ is a $\sigma$-finite disintegration of the $\sigma$-finite measure $\mu$ with respect to $(\nu,\phi)$. Then there exists a sequence, $(D_n)$, from $\mathcal{B}(X)$ such that
\begin{enumerate}
\item $\forall y \in Y$, $\mu_y(D_n) < \infty$
\item for $\nu$-a.e. $y \in Y$, $\mu_y\(X \setminus \bigcup_n D_n\) = 0$.
\end{enumerate}
\end{thm}

\begin{proof}
Define $F: \mathcal{B}(X) \rightarrow \mathcal{B}(Y)$ by
\begin{align*}
    F(B) = \{y \in Y: \mu_y(B) < \infty\}.
\end{align*}
Note that $\forall B \in \mathcal{B}(X)$, $F(B) = \bigcup_n \{y \in Y: \mu_y(B) < n\}$. Thus $F$ does map $\mathcal{B}(X)$ into $\mathcal{B}(Y)$.

Let $(B_n)$ be a sequence from $\mathcal{B}(X)$ such that $\forall n \in \mathbb{N}, \,\, \mu(B_n) < \infty$ and $X = \bigcup_n B_n$. Note that for every $n$ we have that $\mu(B_n) = \int \mu_y(B_n) d\nu(y) < \infty$. Thus $\mu_y(B_n) < \infty$ for $\nu$-a.e. $y$ and therefore $\nu(Y \setminus F(B_n))=0$. Let $E = \bigcap_n F(B_n)$. Note that
\begin{align*}
\nu(Y \setminus E) &= \nu\(Y \setminus \bigcap_n F(B_n)\) = \nu\(\bigcup_n Y \setminus F(B_n)\) \\
&\leq \sum_n \nu(Y \setminus F(B_n)) =0,
\end{align*}
and consequently
\begin{align*}
\mu(X \setminus \phi^{-1}(E)) &= \mu(\phi^{-1}(Y \setminus E)) = \int_{Y \setminus E} \mu_y(X) d\nu(y) = 0.
\end{align*}

For each $n \in \mathbb{N}$ define $D_n = \phi^{-1}(E) \cap B_n$.  For every $y \in E$ we have $\mu_y(D_n) = \mu_y(\phi^{-1}(E) \cap B_n) \leq \mu_y(B_n) < \infty$ and for every $y \in Y \setminus E$ we have $\mu_y(D_n) = \mu_y(\phi^{-1}(E \cap B_n)) \leq \mu_y(\phi^{-1}(E))=0$. Furthermore
\begin{align*}
\mu_y & \(X \setminus \bigcup_n D_n\) = \mu_y \(X \setminus \(\phi^{-1}(E) \cap \bigcup_n B_n\) \) \\
&= \mu_y\(X \setminus \phi^{-1}(E) \cup \(X \setminus \bigcup_n B_n\)\) \\
&\leq \mu_y\(X \setminus \phi^{-1}(E)\) + \mu_y\(X \setminus \bigcup_n B_n\)
= 0 \text{ for } \nu \text{-a.e. } y.
\end{align*}
\end{proof}

\begin{cor}
Suppose $\{\mu_y:y \in Y\}$ is a $\sigma$-finite disintegration of the $\sigma$-finite measure $\mu$ with respect to $(\nu,\phi)$. There exists a uniformly $\sigma$-finite disintegration $\{\hat{\mu}_y : y \in Y\}$ of $\mu$ with respect to $(\nu,\phi)$ such that $\mu_y = \hat{\mu}_y$ for $\nu$-almost every $y \in Y$.
\end{cor}

\begin{proof}
Let $(D_n)$ be the sequence from $\mathcal{B}(X)$ that is constructed in Theorem \ref{almostsf}. Let $N \in \mathcal{B}(Y)$ be such that $\nu(N)=0$ and such that $\mu_y(X \setminus \bigcup_n D_n) = 0$ for every $y \not \in N$. Define $\hat{\mu}_y$ by
\[
\hat{\mu}_y(B) =
\begin{cases}
\mu_y, &\text{if } y \not \in N \\
0, &\text{if } y \in N
\end{cases}
\]
Clearly, $\{\hat{\mu}_y : y \in Y\}$ has the required properties.
\end{proof}
Maharam's question is whether a given $\sigma$-finite disintegration
must be altered in some fashion to be uniformly $\sigma$-finite or is
it automatically already uniform.
In \cite{GM}, it was noted that if each member of a disintegration,
$\mu_y$ is locally finite, then the disintegration is uniformly
$\sigma$-finite. Also, a canonical representation of uniformly
$\sigma$-finite disintegrations was developed. We also point out
that in \cite{Mah3} Maharam showed how spectral representations could
be carried out for uniformly $\sigma$-finite kernels. Whether these
tools can be carried over the kernels that are not necessarily uniform
remains open.

In section \ref{sec:two}, we give two arguments that the continuum
hypothesis implies the answer to Maharam's question is no. We note that
after sending David Fremlin an earlier version of this work where we
used $\mathbf{V}=\mathbf{L}$, but did not discuss the use of CH, he
commented, \cite {Frem}, and may have independently proved, the answer
is no assuming CH. Our first argument for this fact is rather straightforward
whereas the second argument leads to some interesting
infinitary combinatorial questions.

In section \ref{sec:three}, we begin a more detailed investigation of the relation between Maharam's problem and descriptive set theory. In particular, we assume the existence of a ``special''
coanalytic set, a coanalytic set with some specific properties in the product of the Baire space with
itself. This assumption leads to a more descriptive $\sigma$-finite
disintegration which is not uniformly $\sigma$-finite for $X =Y =
\omega^{\omega}.$ Of course, this result extends to any pair of uncountable
Polish spaces.

In section \ref{sec:four}, assuming G\"odel's axiom of
constructibility, $\mathbf{V}=\mathbf{L}$, we show that special
coanalytic sets exist. As the existence of such sets is of perhaps equal interest as Maharam's problem, we present the construction of such a set in some
detail from basic principles. Since our argument involves methods from logic and set theory
that some readers may not be familiar with, we give specific
references to Kunen's book where the necessary background may be
found.

In section \ref{sec:five}, we show that uniformly $\sigma$-finite
kernels are jointly measurable. We don't know whether the converse
holds.

\section{CH implies the answer is no}\label{sec:two}

We give two proofs demonstrating that the answer to Maharam's
question is no assuming $\ch$.  Each proof will involve the construction of a subset of the plane with some specific properties.  
We first show that such a construction is necessary and sufficient for a 
nonuniformly $\sigma$-finite disintegration into purely atomic measures (by a nonuniformly $\sigma$-finite disintegration we 
mean a $\sigma$-finite disintegration which is not uniformly $\sigma$-finite). 

\begin{thm}\label{dislemma}
Let $X$ and $Y$ be Polish spaces, let $\phi:X \rightarrow Y$ be Borel measurable, and let $\{\mu_y: y \in Y\}$ be a family of purely atomic measures each of which is supported on $\phi^{-1}(y)$.  There exist measures $\mu$ on $\mathcal{B}(X)$ and $\nu$ on $\mathcal{B}(Y)$, such that $\{\mu_y:y \in Y\}$ forms a nonuniformly $\sigma$-finite disintegration of $\mu$ with respect to $(\nu,\phi)$ if and only if
    \begin{enumerate}
        \item \label{B} $\forall B \in \mathcal{B}(X)$ the mapping $y \mapsto \mu_y(B)$ is $\mathcal{B}(Y)$-measurable
        \item \label{nB} The set $W=\{(y,x) \in Y \times X: \mu_y(\{x\})>0\}$ is not the union of countably many graphs of Borel functions $f_n:Y \rightarrow X$.
    \end{enumerate}
\end{thm}

\begin{proof}

Suppose conditions (\ref{B}) and (\ref{nB}) are satisfied.  Fix $y_0 \in Y$ and let $\nu$ be the dirac measure concentrated at $y_0$. For each $B \in \mathcal{B}(X)$ define $\mu(B) = \int \mu_y(B) d\nu(y)$.

By (\ref{B}), the measures $\mu_y$ form a disintegration of $\mu$ with respect to $(\nu,\phi)$ into $\sigma$-finite measures supported on the sections $W_y = \{x:\mu_y(\{x\})>0\} \sbset \phi^{-1}(y)$.  This disintegration is not uniformly $\sigma$-finite.  If it were, then by theorem \ref{jointmeas} which is proven later, the mapping $(y,x) \mapsto \mu_y(\{x\})$ is measurable in $Y \times X$. Thus $W$ is a Borel set with countable sections and is a countable union of Borel graphs, contradicting (\ref{nB}).

Now suppose $\{\mu_y:y \in Y\}$ is a nonuniformly $\sigma$-finite disintegration of $\mu$ with respect to $(\nu,\phi)$ into purely atomic measures.  Let $W = \{(y,x): \mu_y(\{x\})>0\}$. Condition (\ref{B}) is satisfied by the definition of a disintegration.

Suppose $W$ fails condition (\ref{nB}) and $f_n:Y \rightarrow X$ is a sequence of Borel functions such that $W= \bigcup_n \{(y,f_n(y)):y \in Y\}$. Since the sections $W_y$ are disjoint, each $f_n$ is one-to-one.  Then $E_n = f_n(Y)$ is a Borel subset of $X$. For every $y, \ \mu_y(E_n) = \mu_y(\{f_n(y)\}) < \infty$ and $\mu_y(X \setminus \bigcup_n E_n)=0$ a contradiction.

\end{proof}

Restating theorem \ref{dislemma} gives the following corollary.
\begin{cor}
A given disintegration into purely atomic measures is uniformly $\sigma$-finite if and only if the set $W=\{(y,x):\mu_y(\{x\})>0\}$ of atoms is a countable union of Borel graphs.
\end{cor}

\begin{thm} Assume $\ch$. There is a $\sigma$-finite
  disintegration which is not uniformly $\sigma$-finite.
\end{thm}

\begin{proof}

Let $X = [0,1]^2$ and $Y = [0,1]$.  Let $\prec$ be a well-ordering of $[0,1]$ into type $\omega_1$.
Let $W = \{(y,(x,y)) \in Y \times X : x \prec y\}$.
Denote the sections of $W$ by the following.

\[
W_y = \{(x,y) \in X : x \prec y\}.
\]
\noindent
Note that $W_y$ is countable for every $y \in Y$.\\

For each Borel $B \subseteq X$ and for each $n$, let $B_n = \{y \in Y : |B \cap W_y| \leq n\}$.  Note that if $y \in B_n$ then $B_n$ contains all predecessors of $y$. Therefore if $B_n \neq Y$ and $y'$ is the least element in $Y \setminus B_n$ then $B_n$ contains all predecessors of $y'$.  Thus $B_n$ is either equal to $Y$ or is countable and must be Borel.\\

By theorem \ref{dislemma}, counting measure on the sections $W_y$ forms a nonuniformly $\sigma$-finite disintegration with respect to the projection map $\pi_2:X \rightarrow Y$.

\end{proof}

We introduce a combinatorial principle $P(\kappa)$ for $\kappa$ an uncountable cardinal.

\begin{defn} \label{comp}
$P(\kappa)$ is the statement that for every sequence $\{ B_\alpha\}_{\alpha<\kappa}$ of sets
$B_\alpha \subseteq \kappa$, and every family $\{ f_{\alpha,n} \colon \alpha <\kappa,\ n \in \omega\}$
of functions $f_{\alpha,n} \colon \kappa \to \kappa$, there is a sequence $\{ S_\alpha\}_{\alpha <\kappa}
\subseteq \sP_{\omega_1}(\kappa)$ of countable subsets of $\kappa$ satisfying:
\begin{enumerate}
\item \label{com1}
$\forall \alpha <\kappa\  \exists \beta <\kappa\ S_\beta \neq \{ f_{\alpha,n}(\beta) \}_{n \in \omega}$.
\item \label{com2}
$\forall \alpha <\kappa \ \forall n \in \omega\   [ \{ \beta < \kappa \colon |S_\beta \cap B_{\alpha} |
=n\} \text{ is countable or co-countable in $\kappa$} ]$.
\end{enumerate}
\end{defn}

\begin{thm} $P(2^\omega)$ implies there is a purely atomic $\sigma$-finite disintegration which is not uniformly $\sigma$-finite.
\end{thm}

\begin{proof} Take $\{ B_\alpha\}_{\alpha< 2^\omega}$ to consist of all Borel sets and take  $\{ f_{\alpha,n} \colon \alpha < 2^\omega,\ n \in \omega\}$ to be the family of all sequences of Borel measurable functions. Then, by theorem~\ref{dislemma}, taking $\mu_\alpha$ to be counting measure on $S_\alpha,$ we have such a disintegration.
\end{proof}

We are interested in the strength of $P(\kappa).$

\begin{thm}[$\zf$] $P(\omega_1)$ holds. In particular, assuming $\ch$ we have $P(2^\omega)$.
\end{thm}

\begin{proof}
Let the $B_\alpha$ and $f_{\alpha,n}$ be as in the hypothesis of $P(\omega_1)$. We define the
countable sets $S_\beta$, $\beta <\omega_1$, as follows. Assume $S_{\beta'}$ has been defined for all
$\beta'<\beta$. We let $S_\beta$ be such that
\begin{enumerate}
\renewcommand{\theenumi}{\roman{enumi}}
\renewcommand{\labelenumi}{(\roman{enumi})}
\item \label{d1} $\min(S_\beta) > \sup_{\beta'<\beta} \sup(S_{\beta'})$.
\item \label{d2}
for all $\beta'<\beta$, if $B_{\beta'}$ is uncountable then $|S_\beta \cap B_{\beta'}| =\omega$.
\item \label{d3}
$S_\beta \nsubseteq \{ f_{\beta,n}(\beta) \colon n \in \omega\}$.
\end{enumerate}

\noindent
Since there are only countably many $\beta'$ less than $\beta$, we can get a countable $S_\beta$
which meets the second requirement above, and adding an extra point will meet the third requirement.
It is now easy to verify the statements of $P(\omega_1)$. Property~(\ref{com1}) of
\ref{comp} follows from (\ref{d3}) above (using $\beta=\alpha$). To see property (\ref{com2}),
fix $B_\alpha$ and $n \in \omega$. If $B_\alpha$ is countable then by (\ref{d1}) above we have that
for large enough $\beta$ that
$S_\beta \cap B_\alpha =\emptyset$, which gives (\ref{com2}). If $B_\alpha$ is uncountable,
then for $\beta>\alpha$ we have $B_\alpha \cap S_\beta$ is infinite. This again gives (\ref{com2}).
\end{proof}

We show that it is consistent that $P(2^\omega)$ fails.

\begin{thm}
Assume $2^\omega=2^{\omega_1}=\omega_2$. Then $P(2^\omega)$ fails.
\end{thm}

\begin{proof}
Let $\kappa$ denote $2^\omega=\omega_2$. We define the sets $B_\alpha$ and functions $f_{\alpha,n}$
witnessing the failure of $P(\kappa)$. Consider the collection of all $\omega$ sequences
$(f_0,f_1,\dots)$ of functions $f\colon \kappa \to \kappa$ which are eventually constant.
Under our hypothesis there are only $\kappa$ many such $\omega$ sequences of functions,
so we may fix the $f_{\alpha,n}$ so that every such sequence occurs as $(f_{\alpha,0}, f_{\alpha,1},
\dots)$ for some $\alpha <\kappa$. For $\alpha$ a successor ordinal let $B_\alpha=\{ \alpha-1\}$.
From our hypothesis we may let $\{ D_\alpha\}$, for $\alpha < \kappa$ a limit ordinal, enumerate all
subsets $D \subseteq \kappa$ of ordertype $\omega_1$. Let $B_\alpha$, for $\alpha$ a limit ordinal,
be given by $B_\alpha=D_\alpha \cup (\sup(D_{\alpha}) ,\kappa)$.

Suppose $\{S_\beta\}_{\beta <\kappa}$ satisfied (\ref{com1}) and (\ref{com2}).
We first claim that for any $\alpha, \beta <\kappa$ there is a $\gamma >\beta$ such that
$S_\gamma \nsubseteq \alpha$.
To see this, suppose $\alpha, \beta$ were to the contrary. For every $\alpha' <\alpha$ we have that
for large enough $\gamma_1, \gamma_2$ that $\alpha' \in S_{\gamma_1} \leftrightarrow \alpha' \in S_{\gamma_2}$.
For otherwise $B_{\alpha'+1}=\{ \alpha'\}$ would violate (\ref{com2}). But this then gives that for
all large enough $\gamma$ that $S_\gamma= S_\gamma \cap \alpha$ is the same. Let $f_n \colon \kappa
\to \kappa$ be such that $S_\beta= \{ f_n(\beta)\}_{n \in \omega}$ for all $\beta <\kappa$. We may assume
that the $f_n$ are eventually constant, since the $S_\beta$ are eventually constant. So, there is
an $\alpha_0 <\kappa$ such that $f_n(\beta)= f_{\alpha_0,n}(\beta)$ for all $n\in \omega$ and $\beta<\kappa$.
This $\alpha_0$ then violates (\ref{com1}). This proves the claim.
We next claim that there is an $\alpha_0<\kappa$ such that for all $\alpha,\beta <\kappa$ there is
a $\gamma >\beta$ such that $\min(S_\gamma-\alpha_0) >\alpha$. Suppose this claim fails.
We construct inductively an increasing sequence $\alpha_\eta$, for $\eta <\omega_1$, such that
for all $\eta<\omega_1$ and all large enough $\gamma$ we have $\alpha_\eta \in S_\gamma$.
This will contradict the fact that all the $S_\gamma$ are countable. Suppose
$\alpha_\eta$ is defined for $\eta < \eta'$. Let $\alpha=\sup \{ \alpha_{\eta}\colon \eta<\eta'\}$.
By the assumed failure of the claim, there is an $\alpha'>\alpha$ such that for
$\kappa$ many $\gamma<\kappa$ we have $\min(S_\gamma-\alpha)<\alpha'$. We may then fix $\alpha_{\eta'}\in
(\alpha,\alpha')$ such that for $\kappa$ many $\gamma$ we have $\alpha_{\eta'} \in S_\gamma$.
As in the proof of the first claim above, (\ref{com2}) implies that for all large enough $\gamma$
that $\alpha_{\eta'} \in S_\gamma$. Thus, we may continue to construct the $\alpha_\eta$ for
all $\eta <\omega_1$, a contradiction. This proves the second claim. Fix $\bar\alpha$
as in the second claim.
From the second claim, we can get  an increasing $\omega_1$ sequence
$\{\gamma_\eta\}_{\eta<\omega_1}$ such that
$\inf(S_{\gamma_\eta}-\bar\alpha) > \sup_{\eta'<\eta}(\sup S_{\gamma_\eta'})$
for all  $\eta <\omega_1$. Let $\alpha_\eta \in S_{\gamma_\eta}-\bar\alpha$ for all $\eta<\omega_1$.
Let $D=\{ \gamma_\eta \colon \eta \text{ is even } \}$. Let $\delta$ be a limit
ordinal such that $B_{\delta}=D \cup (\sup(D), \kappa)$.
Then $A=\{ \beta <\kappa \colon |S_\beta \cap B_{\delta}| =0\}$ and $\kappa-A$ both meet
$\{ \gamma_\eta\colon \eta<\omega_1\}$ in a set of size $\omega_1$, contradicting (\ref{com2}).
\end{proof}

\begin{prob} Is it consistent that CH fails and $P(2^\omega)$ holds?
\end{prob}

\begin{prob} Is it consistent that every $\sigma$-finite disintegration be uniformly $\sigma$-finite?
\end{prob}

\section{Construction of a nonuniformly $\sigma$-finite disintegration
assuming the existence of a special $\mathbf{\Pi}^1_1$ set}\label{sec:three}

In this section, let both $X$ and $Y$ be the Baire space. So, $X = Y
= \omega^\omega$ where $\omega$ has the discrete topology and $X$ and
$Y$ have the product topology. Let $P$ be a closed subset of $X \times Y$ such that $\forall x \in X$,\,
$P_x$ is nonempty and perfect and if $x \neq x',\, P_x \cap P_{x'} =
\emptyset.$ We say $G$ is a special coanalytic set for $P$ provided
$G \subseteq P$ is a $\mathbf{\Pi}^1_1$
set with the following properties:

\begin{enumerate}
\item
$\forall x \in X \,\, |G_x| = \omega_0$,

\item
$G$ is not the union of countably many $\mathbf{\Pi}^1_1$ graphs over
$X$,
\item
for every $n \in \omega$ and for every $B \in \mathcal{B}(Y),\ \{x \in X:
|B \cap G_x| = n\} \in \mathcal{B}(X)$.
\item
there is a nonempty Borel set ( or even perfect) $H \subseteq X$ with such that $G \cap (H \times Y)$ is the union of countably many pairwise disjoint Borel graphs over $H$.
\end{enumerate}

\begin{thm} Let $X = Y = \omega^\omega$.
Let $P = \{((x_i),(y_i)) \in \omega^\omega \times \omega^\omega : \forall i \in \omega [y_{2i} = x_i]\}$.

If $G$ is a special coanalytic set for $P$, then there exists a
$\sigma$-finite measure $\mu$ on $Y$, a $\sigma$-finite measure $\nu$
on $X$, a Borel measurable map $\phi:Y \mapsto X$, and a
$\sigma$-finite disintegration $\{\mu_x:x \in X\}$ of $\mu$ with
respect to $(\nu,\phi)$ which is not uniformly $\sigma$-finite.
\end{thm}

\begin{proof}
Let $\pi_i: \omega^\omega \times \omega^\omega \rightarrow
\omega^\omega$ be the projection map onto the $i$th coordinate. Note
$P$ is closed, $\pi_1(P) = \omega^\omega = \pi_2(P)$, and if
$x,x' \in \omega^\omega$ with $x \neq x'$ then $P_x \cap P_{x'} = \emptyset$.
Note the sections $P_x$ are disjoint and perfect.
 Define the function $\phi:Y \rightarrow X$ by $\phi(y)=x \iff y \in P_x$.
The function $\phi$ is Borel measurable since its graph is a Borel
set. Next define a $\sigma$-finite transition kernel  $\{\mu_x:x \in
X\}$. For each $x \in X$ and $B \in \mathcal{B}(Y)$ define $\mu_x (B)
= |B \cap G_x|$, \emph{i.e.}, counting measure on the fibers of
$G$. Since each fiber $G_x$ is countably infinite, $\mu_x$ is
$\sigma$-finite for all $x$ in $X$. Also since the fibers are pairwise
disjoint, $\mu_x(Y \setminus \phi^{-1}(x)) = 0$.
If $B \in \mathcal{B}(Y)$ then $\{x : \mu_x (B) \geq n\} = \{x : |B
\cap G_x| \geq n\}$ which is a Borel subset of $X$ since $G$ is
special. Thus for every $B \in \mathcal{B}(Y)$ the function $x
\rightarrow \mu_x(B)$ is $\mathcal{B}(X)$-measurable and  $\{\mu_x:x \in
X\}$ is a transition kernel.

Since $G$ is special, there is a Borel set $H \subseteq X$ and Borel functions $f_n : X \rightarrow Y$ with pairwise disjoint graphs such that for every $x \in H \ G_x = \bigcup_n \{f_n(x)\}$. Note that since the sections of $G$ are pairwise disjoint, each $f_n$ is 1-to-1 over $H$. Let $\nu$ be a probability measure on $\mathcal{B}(X)$ such that $\nu(H)=1$.

Define a measure $\mu$ on the Borel subsets of $Y$ by
\begin{align*}
\mu(B) = \int \mu_x(B) d\nu(x).
\end{align*}

We first show that $\mu$ is $\sigma$-finite.  Let $B_n = f_n(H)$ and note that $\forall x \in H$, $G_x \sbset \bigcup_n B_n$. Each $B_n$ is Borel since each $f_n$ is 1-to-1 over $H$, and $\forall x \in H$, $\mu_x(B_n) = |B_n \cap G_x| = 1$. Furthermore
\begin{align*}
\mu &\( Y \setminus \bigcup_n B_n \) = \int \mu_x \(Y \setminus \bigcup_n B_n \) d\nu(x) \\
&= \int_{X \setminus H} \left|\( Y \setminus \bigcup_n B_n \) \cap G_x \right| d\nu(x) + \int_H \left|\( Y \setminus \bigcup_n B_n \) \cap G_x \right| d\nu(x)\\
&= \int_H \left|\(Y \setminus \bigcup_n B_n\) \cap G_x \right| d\nu(x) = 0.
\end{align*}

The measure $\mu$ is thus a $\sigma$-finite measure on $Y$ and the family $\{ \mu_x : x \in X\}$ is a disintegration of $\mu$ with respect to $(\nu,\phi)$ into $\sigma$-finite measures.  However, this disintegration cannot be uniformly $\sigma$-finite.  If it were, there would exist countably many Borel sets $E_n \sbset Y$ such that $\forall x \in X$, $\mu_x(E_n) < \infty$ and $\mu_x (Y \setminus \cup_n E_n) = 0$.  Thus for each $x \in X$, $|G_x \cap E_n| < \infty$ and $G \sbset \bigcup_n X \times E_n$. For each $n$, $G \cap (X \times E_n)$ is $\mathbf{\Pi}^1_1$ with finite sections and is thus a countable union of $\mathbf{\Pi}^1_1$ graphs (see \cite{JM}) implying that $G = \bigcup_n G \cap E_n$ is a countable union of $\mathbf{\Pi}^1_1$ graphs, a contradiction.
\end{proof}

This argument shows that in fact there does not exist countably many $E_n \in \mathcal{B}(X \times Y)$ satisfying $\forall x \,\, \mu_x(E_{nx}) < \infty$ and $\mu_x (Y \setminus \bigcup_n E_{nx}) = 0$.

\section{Construction of a ``special'' $\mathbf{\Pi}^1_1$ set assuming $\mathbf{V}=\mathbf{L}$} \label{sec:four}

In this section we consider the Polish spaces $X = Y = \omega^\omega$ and we prove the existence of a ``special'' $\mathbf{\Pi}^1_1$ set assuming $\mathbf{V}=\mathbf{L}$.  In order to do this we first put in place the formal logical structures which will be needed.  We let $\text{ZF}_N$ denote a finite fragment of $\text{ZF}$ that is large enough such that $\Pi^{1}_{1}$ and $\Sigma^1_1$ formulas are absolute for transitive models of $\text{ZF}_N$.

It will be necessary to code models by elements of $\omega^\omega$.  We now make this coding specific.

For each $n$ let $\phi_n$ be the $n$-th formula in the G\"{o}del
numbering of the formulas in the language $\mathcal{L}^\in$ (see
\cite{Kun} Def 1.4 pp 155). Given $x \in \{0,1\}^\omega \sbset
\omega^\omega$, we will define the theory $Th_x$ by $\phi_n \in Th_x$
if and only if $x(n) = 1$. Let $\phi_{<L}$ be a formula defining the
canonical well-ordering of $\mathbf{L}$ and let $M \in \omega$ be the
integer such that $\phi_M = ``\phi_{<L}  \ \text{\ is a well ordering
of the universe}."$

Let $C \sbset \omega^\omega$ be the collection of codes of theories, \emph{i.e.}, $x \in C$ iff:
\begin{enumerate}
    \item $x \in \{0,1\}^\omega$
    \item $Th_x$ is a consistent and complete theory of $\text{ZF}_N + (\mathbf{V} = \mathbf{L})$
    \item $x(M) = 1$.
\end{enumerate}
Note that $C$ is a $\mathbf{\Delta}^1_1$ set.

Given a formula $\phi_n (w,x_1, \ldots, x_k)$ with free variables 
$w,x_1, \ldots, x_k$ define the \textbf{Skolem term} for $\phi_n$ to
be the 
corresponding formula $\tau_n(z,x_1,\ldots,x_k)$ where $\tau_n(z,x_1,\ldots,x_n)$ is
\begin{align*}
&\left( \exists w \, \phi_n(w,x_1,\ldots,x_k) \wedge z \text{ is the} <_L \text{least such } w \right) \vee \\
&\left( \neg \exists w \, \phi_n(w,x_1,\ldots,x_k) \wedge z = 0 \right).
\end{align*}


For each $x \in \{0,1\}^\omega$ if $S$ is a collection of Skolem terms, define an equivalence relation, $\equiv_x$, on $S$ by
\[
\tau_n \equiv_x \tau_m \iff Th_x \vdash \tau_n = \tau_m.
\]

For $x \in C$, define $M_x$ to be the set of equivalence classes of
all Skolem terms arising from formulas $\phi(w)$ such that $Th_x
\vdash \exists w[\phi(w)]$. We note the Skolem hull of $\emptyset$ inside of $M_x$ is all of $M_x$. In other words, $M_x$ is the smallest model of the theory $Th_x$. Define the relation $E_x$ on $M_x \times M_x$ by
\[
[\tau_i] E_x [\tau_j] \iff Th_x \vdash \tau_i\in \tau_j.
\]

Recall that a structure $M$ with binary relation $E$ is well-founded if every subset of $M$ contains an $E$-minimal element (see \cite{Kun} Ch. 3). For each $x \in C$, note that $M_x$ does not necessarily code a well-founded structure. However, if $M_x$ is well-founded, then there exists a countable ordinal $\alpha$ such that $M_x \cong L_{\alpha}$ (see \cite{Kun} Thm. 3.9(b) p. 172).  The following proposition shows that codings of well-founded models are unique.

\begin{prop}
Suppose $x,x' \in C$ and there is an ordinal $\alpha$ such that $M_x \cong L_{\alpha} \cong M_{x'}$. Then $x = x'$.
\end{prop}

\begin{proof}
Let $T$ be the theory of $L_{\alpha}$.
Since $M_x \cong L_{\alpha}$ and $M_{x'} \cong L_{\alpha}$, both $x$ and $x'$ code $T$. Then for every $n$, $x(n) = 1 \iff \phi_n \in T \iff x'(n)=1$.  Thus $x = x'$.
\end{proof}

We next show that if an element $w$ of $\omega^\omega$ is constructed at an ordinal $\alpha$ then there exists a code $x \in C$ for a structure $(M_x,E_x)$ that is isomorphic to $L_{\alpha}$.
\begin{prop}
If $w \in \omega^\omega \cap L_{\alpha + 1} \setminus L_{\alpha}$ then $\exists x \in C$ such that $M_x \cong L_{\alpha}$.
\end{prop}

\begin{proof}
Let $T$ be the theory of $L_{\alpha}$ and let $x \in C$ such that $Th_x = T$. Then $(M_x,E_x)$ is an elementary submodel of $(L_{\alpha},\in)$ (see \cite{Kun} Lemma 7.3 p.136). Since $L_{\alpha}$ is well-founded, $M_x$ is well-founded. Then $\in$ is well-founded on the transitive collapse $TC(M_x)$ (see \cite{Kun} Thm. 5.14 p. 106) and thus $(M_x,E_x) \cong (TC(M_x),\in) \cong (L_{\beta},\in)$ for some $\beta \leq \alpha$. So $w \in L_{\beta + 1}$ and thus $\beta = \alpha$.
\end{proof}

\begin{thm} \label{JMthm}
Assume $\mathbf{V}=\mathbf{L}$. Let $X = Y = \omega^\omega$. Let $P$ be a closed subset of $X \times Y$ such that $\forall x \in X$,\,
$P_x$ is nonempty and perfect and if $x \neq x',\, P_x \cap P_{x'} = \emptyset.$ Then there exists a $\mathbf{\Pi}^1_1$ set $G \sbset P$ with the following properties:

\begin{enumerate}
\item
$\forall x \in X, |G_x| = \omega_0$
\item
For every $n \in \omega$ and for every $\mathbf{\Delta}^1_1$ set $B \sbset Y,\, \{x \in X:
|B \cap G_x| \geq n\}$ is $\mathbf{\Delta}^1_1$
\item
G is not the union of countably many $\mathbf{\Pi}^1_1$ graphs over $X$.
\item
There is a nonempty $\mathbf{\Delta}^1_1$ (or even perfect) set $H \subseteq X$ such that $G \cap (H \times Y)$ is the union of countably many pairwise disjoint $\mathbf{\Delta}^1_1$ graphs over $H$.
\end{enumerate}
\end{thm}

\begin{proof}
Fix a pair of recursive bijections, $x \mapsto (x^n)_{n=0}^{\infty}$ from $\omega^\omega$ onto $(\omega^\omega)^\omega$ and $x \mapsto (x^0,x^1)$ from $\omega^\omega$ onto $\omega^\omega \times \omega^\omega$. Denote the inverse of the second bijection by $(y,z) \mapsto \langle y,z \rangle$.
Call an ordinal $\beta$ \textbf{good}  if $L_{\beta} \models \text{ZF}_N
+ (\mathbf{V}=\mathbf{L})$.

Let $p \in \omega^\omega$ be a code for $P$. In this regard, when we say ``$z$ codes the Borel set $B$'' we mean 
a coding such that the statement ``$w$ is in the set coded by $z$'' is absolute to all transitive models
of $\zf_N$ (for example, we could have $z$ code a wellfounded tree on $\omega$ which gives an inductive construction of $B$
from the basic open sets).

For each $n \in \omega$ let $f_n : X \rightarrow Y$ be a $\mathbf{\Delta}^1_1$
function such that $\forall x \in X$ and for $n \neq m$ $f_n(x) \neq f_m(x)$ and such that $\forall x \in X \,\, \forall n \in \omega \,\, f_n(x) \in P_x$.

For a given $w \in \omega^\omega$ and an $x \in C$ coding an $\omega$-model $M_x$ (\emph{i.e.} $\omega$ is in the well-founded part of $M_x$), we will make frequent use of the shorthand ``$w \in M_x$'' to mean (for convenience, we identify here $\omega^\omega$ with $\sP(\omega)$)
\[
\exists \tau \in \text{dom}(M_x) \,\, [(M_x \models ``\tau \sbset \omega\text{''}) \wedge TC(\tau) = w].
\]

Define $U \sbset C$ by $x \in U$ if and only if there exists an ordinal $\alpha(x) \geq \omega_0$ such that $M_x \cong L_{\alpha(x)}$ and $p \in L_{\alpha(x)}$. Define $V \sbset C$ by $x \in V$ iff $M_x$ is an
$\omega$-model, and ``$p \in M_x$''.  Note that $U \sbset V$, $V$ is $\mathbf{\Delta}^1_1$, and that the elements of $U$ code well-founded structures.

Define the set $G' \sbset X \times Y$ by $(x,y) \in G' \iff$
\begin{align*}
[x &\not \in V \wedge \exists n (y = f_n(x))] \vee
[x \in V \wedge (x,y) \in P \wedge [``y \in M_x \text{"} \vee \\
&\exists \text{ a well-founded extension } M \text{ of } M_x \,\, \exists \alpha', \alpha < \omega_1 \\
&\quad (L_{\alpha'} \cong M_x \sbset M \cong L_{\alpha} \wedge y \in L_{\alpha} \wedge \\
&\quad [\forall \alpha' \leq \gamma  < \alpha \ (\neg (\gamma \text{ is good and a limit of good ordinals}) \vee\\
&\qquad \exists \phi \in \Sigma^1_2 \,\, \exists \tau > \gamma \, (L_{\gamma} \models \neg \phi \wedge L_{\tau} \models  \phi))])]].
\end{align*}

To clarify, if $x \in V$ and $M_x$ is ill-founded then $G'_x$ consists of all reals in $M_x$.  If $x \in V$ and $M_x$ is well-founded then we continue adding reals to the section $G'_x$ until the truth of $\Sigma^1_2$ statements stabilize to be true.

Note that $G'$ is $\Sigma^1_2$ and let $\Omega'(x,y)$ be the above $\Sigma^1_2$ formula defining $G'$.

We first show that the sections of $G'$ are countable.  Clearly $G'_x$ is countable for every $x \not \in V$. 
Since each model $M_x$ is countable, $G'_x$ is countable for every $x \in V \setminus U$. Finally suppose $x \in U$. 
Let $M$ be a well-founded extension of $M_x$ as in the definition above for $G'$. Let $\alpha$ be the ordinal such that $M \cong L_{\alpha}$. 
Let $\beta$ be the least good ordinal less than $\omega_1$ such that $L_{\beta}$ is a $\Sigma_2$ 
elementary substructure of $\mathbf{L}$. Then for every $\beta' > \beta$ and every $\Sigma^1_2$ formula $\phi$ we have
\[
L_\beta \models \phi \iff L_{\beta'} \models \phi \iff \mathbf{L} \models \phi.
\]
We clearly have that $\beta$ is good and a limit of good ordinals, and by the definition of $G'$ we must have $\beta \geq \alpha$. 
Thus $G'_x \sbset L_\beta$ and is therefore countable.

Let $G$ be a ${\Pi}^1_1$-uniformization of $G'$, \emph{i.e.} a subset of $\omega^\omega \times \omega^\omega$ such that for every $x \in \omega^\omega$
\[
G'(x,y) \iff \exists z \, G(x, \langle y,z \rangle ) \iff \exists ! z \, G(x,\langle y,z \rangle ).
\]
Let $\Omega$ be a $\Pi^1_1$ formula defining $G$. We assume that $\text{ZF}_N$ was chosen large enough such that the following is a theorem of $\text{ZF}_N$.
\[
\forall x \, \forall y [\Omega'(x,y) \iff \exists z \, \Omega(x,\langle y,z \rangle ) \iff \exists ! z \, \Omega(x,\langle y,z \rangle )].
\]

Note that since the sections of $G'$ are countable so too are the sections of $G$. Note also that if $H = X \setminus V$ then property $(4)$ holds for $G$. Next we proceed to show that the Borel condition in property $(2)$ holds for $G$.

Fix a $\mathbf{\Delta}^1_1$ set $B
\sbset Y$, fix an $n \in \omega$, let $K_n = \{x \in X : |B \cap G_x| \geq
n\}$, let $b \in \omega^\omega$ be a code for $B$, and since we are assuming $\textbf{V} = \textbf{L}$ let $\tau$ be the level of $L$ at which $b$ is constructed. Then $\tau$ is well-defined and $\tau < \omega_1$. Partition $V$ into the following $\mathbf{\Delta}^1_1$ sets: $E = \{x \in V : ``b \not \in M_{x}\text{''}\}$ and $D = \{x \in V : ``b \in M_{x}\text{''}\}$.

Define the formula
\[
    \psi(x) = \exists \text{ distinct } a_1,\ldots,a_n\  [\text{``} a_1,\ldots,a_n \in M_x \text{''} \wedge
            (a_1, \ldots, a_n \in B) ]
\]

\noindent
Clearly $\psi(x)$ is a $\Sigma^1_1$ statement about $x$.

By the definition of $G$, $\psi$ correctly defines $K_n$ on $V \setminus U$. For $x \in U \cap D$, ``$b \in M_x$'' and since $\Sigma^1_1$ statements are absolute between transitive models of $\text{ZF}_N$, $\psi$ correctly defines $K_n$ on $U \cap D$. 
Since $\tau < \omega_1$ and distinct $x \in U$ determine distinct well-founded $L_{\alpha}$, there can be only countably many $x \in U$ which code $L_\alpha$ with $\alpha < \tau$. If $x \in U \cap E$ then $M_x \cong L_{\alpha}$ where $\alpha < \tau$.  Thus $U \cap E$ is countable.  Therefore the formula $\psi$ correctly defines $K_n$ on $V$ except for the countable set $U \cap E$.

To see that $(K_n \cap V) \setminus (U \cap E)$ is $\Delta^1_1$, note that the formula $\psi$ is equivalent to the $\Sigma^1_1$ formula
\begin{align*}
\exists i_1, &\ldots, \exists i_n \in \omega, \exists a_1, \ldots, \exists a_n \in \omega^\omega \,\,
    [M_x \models ``i_1, \ldots, i_n \in \omega^\omega \text{''} \wedge \\
    &TC(i_1)=a_1, \ldots, TC(i_n)=a_n \, \wedge
    \, a_1, \ldots, a_n \in B]
\end{align*}
which is equivalent to the $\Pi^1_1$ formula
\begin{align*}
\exists i_1, &\ldots, \exists i_n \in \omega, \forall a_1, \ldots, \forall a_n \in \omega^\omega \,\,
    [M_x \models ``i_1, \ldots, i_n \in \omega^\omega \text{''} \wedge \\
    &\, \left(TC(i_1)=a_1, \ldots, TC(i_n)=a_n \right) \Rightarrow a_1 \in B, \ldots, a_n \in B ].
\end{align*}
Thus $\psi$ defines a $\Delta^1_1$ set which gives $K_n$ on $V \setminus (U \cap E)$, and since $(U \cap E)$ is countable, $K_n \cap V$ is $\Delta^1_1$.

For $x \in X \setminus V$ each section $G_x = \bigcup_{n=1}^{\infty} f_n(x)$. Thus for each $x \in X \setminus V$ we have
\begin{align*}
|B \cap G_x| \geq n &\iff
\exists \text{ distinct}\ k_1, \ldots, k_n [f_{k_1}(x) \in B, \ldots, f_{k_n}(x) \in B] \\
&\iff x \in \bigcup_{(k_1,\ldots,k_n)} f_{k_1}^{-1}(B) \cap \ldots \cap f_{k_n}^{-1}(B).
\end{align*}
Therefore $K_n \cap X \setminus V$ is $\mathbf{\Delta}^1_1$.

Finally we show that property (3) holds for $G$.
Proceeding by contradiction suppose that $G$ could be written as a countable union of $\mathbf{\Pi}^1_1$ graphs $G_m$. Choose a sequence $(x^m)$ from $\omega^\omega$ and formulas $\psi_m (x,y)$ so that $\psi_m$ are $\mathbf{\Pi}^1_1 (x^m)$ formulas defining the $G_m$.  Let $x' \in \omega^\omega$ be such that $x'$ codes the sequence $(x^{m})_{m=0}^{\infty}$ and choose $x \in U$ and $\alpha$ such that $M_x \cong L_{\alpha}$ 
and $x' \in L_{\alpha}$. Next let $\beta \geq \alpha$ be the least ordinal such that $(\beta \text{ is good and a limit of good ordinals}) 
\wedge \forall \phi \in \Sigma^1_2 \, (L_{\beta} \models \neg \phi \Rightarrow \forall \tau > \beta \, L_{\tau} \models \neg \phi)$.

From the definition of $G'$ we have that $\omega^\omega \cap L_{\beta} \sbset G'_x$. Furthermore if $y \in L_{\beta}$ 
then for some good ordinal $\delta < \beta$, $y \in L_{\delta}$. 
Since $\beta$ was chosen to be minimal, we have that $\forall \gamma < \delta \, [\neg(\gamma$ is good and a limit of good ordinals $) \vee \exists \phi \in \Sigma^1_2 \,\,(L_{\delta} \models \neg \phi \wedge \exists \tau > \gamma\,  (L_{\tau} \models \phi))]$. 
In fact we may replace ``$\exists \tau > \gamma$'' in the previous statement with ``$\exists \tau > \gamma, \tau < \beta$''. 
Thus $\delta$ witnesses that $L_\beta \models \Omega'(x,y)$.

Since $\beta$ was chosen so that $\Sigma^1_2$ statements are stabilized at $\beta$, 
we have that $L_{\beta} \models ``\{y: \exists m \, \psi_m (x^m,y)\}$ is countable''. However, $L_{\beta} \models ``\omega^\omega$ is uncountable''. Thus we may let $y,z \in L_{\beta}$ such that
\begin{align*}
L_{\beta} &\models \Omega(x, \langle y,z \rangle ) \text{ and}\\
L_{\beta} &\models \forall m \,\, \neg \psi_m (x^m,\langle y,z \rangle ).
\end{align*}
Then by absoluteness $\mathbf{L} \models \forall m \, \neg \psi(x^m,\langle y,z \rangle )$. 
Thus $\forall m\, (x,\langle y,z \rangle ) \not \in G_m$. However this contradicts the fact that 
$\mathbf{L} \models \Omega(x,\langle y,z \rangle )$ by absoluteness and therefore $(x,\langle y,z \rangle ) \in G$.

\end{proof}

This naturally leads us to ask:
\begin{prob}
Can one show in ZFC that special $\mathbf{\Pi}^1_1$ sets exist?
\end{prob}

\section{Uniformly $\sigma$-finite implies joint measurability but
  does the  converse hold} \label{sec:five}

Let $(X,\mathcal{B}(X))$ and $(Y, \mathcal{B}(Y))$ be Polish spaces,
let $\phi\colon X \rightarrow Y$ be $\mathcal{B}$-measurable and let $\mu$
and $\nu$ be measures on $\mathcal{B}(X)$ and $\mathcal{B}(Y)$. 
Let $y \mapsto \mu_y$ be a {\em measure kernel}, that is, each $\mu_y$ is a measure
on the Borel subsets of $X$ and such that for each Borel set $E$ in
$X$, the map $y \mapsto \mu_y(E)$ is Borel measurable 
(this is part of the definition of a disintegration).
Let $\mathcal{K}(X)$ be the space of compact subsets of $X$ equipped with
the Vietoris topology or equivalently the topology generated by the
Hausdorff metric.

\begin{lem}\label{meas} If for every $y$,  $\mu_y(X) < \infty,$ then the map
  $F\colon Y\times\mathcal{K}(X)\mapsto \R$, given by $F(y,K) = \mu_y(K)$, is Borel measurable.
\end{lem}

\begin{proof}
Fix a basis for the topology of $X$, say $\{V_n\}_{n=1}^\infty.$
Enumerate sets of the form $\{K\colon K\subseteq V_{i_1}\cup\ldots\cup V_{i_j}\}$,
say  $\{U_n\}_{n=1}^\infty$. $U_n$ is an open set in $\mathcal{K}(X)$, and let 
$\tilde{U}_n=V_{i_1}\cup\ldots\cup V_{i_j}$ be the corresponding open set in $X$. 
Fix a number $c$. For each $n$, let $D_n
=\{y\colon \mu_y(\tilde{U}_n) < c\}$. We have $\{(y,K)\colon \mu_y(K) < c\} =\bigcup_n (D_n
\times U_n)$.
\end{proof}

\begin{lem}\label{approx} If for every $y, \ \mu_y(X) < \infty,$ then for each
  $\epsilon > 0$, there is a Borel measurable map $y \mapsto K \in
  \mathcal{K}(X)$ such that for every $y, \  \mu_y(X\setminus K_y) < \epsilon.$
\end{lem}
\begin{proof}
This lemma follows from Theorem 2.2 of \cite{M}.
\end{proof}

\begin{thm}\label{jointmeas}
Suppose $\{\mu_y : y \in Y\}$ is a $\sigma$-finite disintegration of
$\mu$ with respect to $(\nu,\phi)$. Consider the following statements.
\begin{enumerate}
    \item $\{\mu_y : y \in Y\}$ is uniformly $\sigma$-finite.
    \item There is a sequence of Borel mappings $y \mapsto K_n(y)$ from $Y$ into $\mathcal{K}(X)$ satisfying
        \begin{itemize}
            \item $\forall y \, \forall n \,\, \mu_y(K_n(y)) < \infty$
            \item $\forall y \,\, \mu_y(X \setminus \bigcup_n K_n(y)) = 0$.
        \end{itemize}
    \item The mapping $(y,K) \mapsto \mu_y(K)$ from $Y \times
      \mathcal{K}(X) \rightarrow \mathbb{R}$ is Borel measurable.
\end{enumerate}
Then statements (1) and (2) are equivalent and each them implies statement (3). Moreover, if each measure
$\mu_y$ is purely atomic, then statements (1),(2), and (3) are equivalent.
\end{thm}

\begin{proof}

(1) $\Rightarrow$ (2) \\
Fix $\{B_n\}$ witnessing the kernel $y\mapsto \mu_y$ is uniformly
$\sigma$-finite. We may and do assume that for each $n, B_n \subseteq
B_{n+1}.$ For each $n$, let $\mu_{ny}(E) = \mu_y(E\cap B_n)$.
then by Lemma~\ref{approx}, we obtain Borel measurable maps $y\mapsto
K_{nmy}\in \mathcal{K}(X)$ such that for every $y, \ \mu_{ny}(X
\setminus \bigcup_m K_{nmy}) = 0.$ The implication follows.

(1) $\Rightarrow$ (3) \\
Continuing with the preceding argument, we see that for each $n$, the
map $F_n(y,K) = \mu_y(B_n\cap K)$ is Borel measurable and $F_n(y,K)$
converges up to $F(y,K)$.

(2) $\Rightarrow$ (1) \\
For each $n$ let $G_n$ be the `epigraph' of the mapping $y \mapsto K_n(y)$. By `epigraph' we mean
\[ G_n = \{(y,x) : x \in K_n(y)\}. \]

Note that a function $f \colon Y \rightarrow \mathcal{K}(X)$ is Borel 
iff the epigraph, $\{(y,x) \colon x \in f(y)\}$ is Borel in $Y \times X$.

Let $B_n = \pi_X (G_n \cap Graph(\phi))$.  This projection is 1-to-1 therefore $B_n$ is Borel.  Observe that
\begin{align*}
    \mu_y(B_n) &= \mu_y(K_n(y) \cap \phi^{-1}(y)) \\
        &= \mu_y(K_n(y)) < \infty \quad \text{and} \\
    \mu_y(X \setminus \bigcup_n B_n) &=
        \mu_y \(X \setminus \bigcup_n (K_n(y) \cap \phi^{-1}(y))\) \\
        &= \mu_y \(X \setminus \bigcup_n K_n(y)\) = 0.
\end{align*}

Finally, let us assume that for every $y$, the measure $\mu_y$ is
purely atomic and statement (3) holds. Let $W = \{(y,K)\colon \mu_y(K) > 0
\ \text{and} \ \card(K) = 1\}.$ Then $W$ is a Borel subset of $Y\times
\mathcal{K}(X)$ with countable sections. Therefore, there are Borel
functions $y \mapsto\mathcal{K}(X)$ whose graphs fill up $W$. This
means statement (2) holds.

\end{proof}

\begin{prob}
Is it true that a disintegration is uniformly $\sigma$-finite if and
only if the map $(y,K)\mapsto \mu_y(K)$ is jointly measurable?
\end{prob}

We would like to mention the following problem concerning the mixture
operator defined by a measure transition kernel.

\begin{prob} Suppose we are given
a measure kernel $y \mapsto \mu_y$ (defined at the beginning of this section). 
Consider the mixture operator $T$ defined by
$$
T(\lambda)(E) := \int_Y \mu_y(E)d\lambda(y).
$$
 Suppose this operator has the property that it maps $\sigma$-finite
 (signed) measures on $Y$ to $\sigma$-finite (signed) measures on $X$
 and the operator $T$ is lattice preserving, i.e., T takes mutually
 singular measures to mutually singular measures. Is there a
 universally measurable map $\phi:X \mapsto Y$ such that for each $y,\
 \mu_y(X \setminus \phi^{-1}(y)) = 0?.$
\end{prob}

We mention that it was shown in \cite{MPW} that the answer is yes
assuming Martin's axiom or even weaker that a medial limit exists
provided for each $y,\ \mu_y$ is a probability measure.

\end{document}